\documentclass[oneside,12pt]{amsart}
\usepackage{geometry}
\usepackage{enumitem,verbatim}
\usepackage{microtype}
\usepackage{physics}
\usepackage[T1]{fontenc}

\usepackage{hyperref}
\usepackage{color}
\hypersetup{colorlinks,linkcolor=blue,citecolor=cyan}
\usepackage{amsthm}
\usepackage[nameinlink]{cleveref}
\usepackage{amsmath,amssymb}

\usepackage{thmtools, thm-restate}

\newcommand{\mb}{\mathbb}
\newcommand{\mc}{\mathcal}
\newcommand{\ol}{\overline}

\DeclareMathOperator{\Hom}{Hom}
\DeclareMathOperator{\MALG}{MALG}

\newcommand{\hra}{\hookrightarrow}
\newcommand{\B}{\mc B}
\newcommand{\N}{\mb N}
\newcommand{\R}{\mb R^+}

\newcommand{\MvN}{\mathrm{MvN}}

\newtheorem{thm}{Theorem}[section]
\newtheorem{cor}[thm]{Corollary}

\newtheorem{prop}[thm]{Proposition}
\theoremstyle{definition}
\newtheorem{defn}[thm]{Definition}
\newtheorem{eg}[thm]{Example}

\newtheorem{rmk}[thm]{Remark}

\title{Equidecomposition in cardinal algebras}

\author{Forte Shinko}
\address{Division of Physics, Mathematics and Astronomy,
California Institute of Technology,
1200 E. California Blvd.,
Pasadena,
CA 91125}
\email{fshinko@caltech.edu}
\thanks{The author was partially supported by NSF Grant DMS-1464475.}

\date{\today}

\begin{document}
\maketitle

\begin{abstract}
  Let $\Gamma$ be a countable group.
  A classical theorem of Thorisson states that
  if $X$ is a standard Borel $\Gamma$-space and $\mu$ and $\nu$ are Borel probability measures on $X$
  which agree on every $\Gamma$-invariant subset,
  then $\mu$ and $\nu$ are equidecomposable,
  i.e. there are Borel measures $(\mu_\gamma)_{\gamma\in\Gamma}$ on $X$
  such that $\mu = \sum_\gamma \mu_\gamma$ and $\nu = \sum_\gamma \gamma\mu_\gamma$.
  We establish a generalization of this result to cardinal algebras.
\end{abstract}

\section{Introduction}
In this paper,
$\Gamma$ will always denote a countable discrete group.
Let $X$ be a standard Borel $\Gamma$-space.
A classical theorem of Thorisson \cite{Tho96} in probability theory states that if $x$ and $x'$ are random variables on $X$,
then the distributions of $x$ and $x'$ agree on the $\Gamma$-invariant subsets of $X$
iff there is a shift-coupling of $x$ and $x'$,
i.e. a random variable $\gamma$ on $\Gamma$ such that $\gamma x$ and $x'$ are equal in distribution.
This characterization of shift-coupling has been applied to various areas of probabilty theory
including random rooted graphs \cite{Khe18}, Brownian motion \cite{PT15}, and point processes \cite{HS13}.

This theorem can be reformulated measure-theoretically as follows.
Let $X$ be a standard Borel $\Gamma$-space and let $\mu$ and $\nu$ be Borel probability measures on $X$.
Then $\mu$ and $\nu$ agree on every $\Gamma$-invariant set iff either of the following hold:
\begin{enumerate}
  \item There is a Borel probability measure $\lambda$ on $\Gamma\times X$ such that
    $s_*\lambda = \mu$ and $t_*\lambda = \nu$,
    where $s,t:\Gamma\times X\to X$ are the maps $s(\gamma,x) = x$ and $t(\gamma,x) = \gamma x$,
  \item There is a Borel probability measure $\lambda$ on the orbit equivalence relation
    $E^X_G := \{(x,y)\in X^2:\exists\gamma[x = \gamma y]\}$ such that $s_*\lambda = \mu$ and $t_*\lambda = \nu$,
    where $s,t:E\to X$ are the maps $s(x,y) = x$ and $t(x,y) = y$
    (see \cite[Theorem 1']{Khe18}).
\end{enumerate}
By setting $\mu_\gamma$ to be the measure on $X$ defined by $\mu_\gamma(A) := \mu(\{\gamma\}\times A)$,
we see that $\mu$ and $\nu$ agree on every $\Gamma$-invariant set
iff they are \textbf{equidecomposable},
i.e. there are Borel measures $(\mu_\gamma)_{\gamma\in\Gamma}$ on $X$
such that $\mu = \sum_\gamma \mu_\gamma$ and $\nu = \sum_\gamma \gamma\mu_\gamma$.
In this paper, we show that this statement is an instance of
a more general result about groups acting on (generalized) cardinal algebras,
a concept introduced by Tarski in \cite{Tar49},
leading to a purely algebraic proof of the statement.

A \textbf{generalized cardinal algebra (GCA)} is a set $A$ equipped with a partial binary operation $+$,
a constant $0$, and a partial $\omega$-ary operation $\sum$ subject to the following axioms,
where we use the notation $\sum_n a_n = \sum(a_n)_n$:
\begin{enumerate}
  \item
    If $\sum_n a_n$ is defined,
    then
    \[
      \sum_n a_n = a_0 + \sum_{n\ge 1} a_n.
    \]
  \item
    If $\sum_n (a_n + b_n)$ is defined,
    then
    \[
      \sum_n (a_n + b_n)
      = \sum_n a_n + \sum_n b_n.
    \]
  \item For any $a\in A$,
    we have $a + 0 = 0 + a = a$.
  \item (Refinement axiom)
    If $a + b = \sum_n c_n$,
    then there are $(a_n)_n$ and $(b_n)_n$ such that
    \[
      a = \sum_n a_n, \qquad
      b = \sum_n b_n, \qquad
      c_n = a_n + b_n.
    \]
  \item (Remainder axiom)
    If $(a_n)_n$ and $(b_n)_n$ are such that $a_n = b_n + a_{n+1}$,
    then there is $c\in A$ such that for each $n$,
    \[
      a_n = c + \sum_{i\ge n} b_i.
    \]
\end{enumerate}
These axioms imply in particular that $\sum$ is commutative:
if $\sum_n a_n$ is defined and $\pi$ is a permutation of $\N$,
then $\sum_n a_n = \sum_n a_{\pi(n)}$
(see \cite[1.38]{Tar49}).

A \textbf{cardinal algebra (CA)} is a GCA whose operations $+$ and $\sum$ are total.
Cardinal algebras were introduced by Tarski in \cite{Tar49}
to axiomatize properties of ZF cardinal arithmetic,
such as the cancellation law $n\cdot\kappa = n\cdot\lambda \implies \kappa = \lambda$.
More recently, they have been used in \cite{KM16} in the study of countable Borel equivalence relations.

Some examples of GCAs and CAs are as follows.
\begin{itemize}
    \item \cite[14.1]{Tar49}
        $\N$ and $\R$ are GCAs under addition,
        where $\R$ is the set of non-negative real numbers.
    \item \cite[2.1]{Sho90}
        If $X$ is a measurable space,
        then the set of measures on $X$ is a CA under pointwise addition.
    \item \cite[15.10]{Tar49}
        Every $\sigma$-complete,
        $\sigma$-distributive lattice is a CA under join.
        In particular, for any set $X$,
        the power set $\mc P(X)$ is a CA under union.
    \item \cite[17.2]{Tar49}
        The class of cardinals is a CA under addition
        (although strictly speaking, we require a CA to be a set).
    \item \cite[15.24]{Tar49}
        Every $\sigma$-complete Boolean algebra is a GCA under join of mutually disjoint elements.
        \begin{itemize}
            \item If $X$ is a measurable space,
                then the collection $\B(X)$ of measurable sets is a GCA under disjoint union.
            \item If $(X,\mu)$ is a measure space,
                then the measure algebra $\MALG(X,\mu)$ is a GCA under disjoint union.
        \end{itemize}
\end{itemize}

Every GCA is endowed with a relation $\le$ given by
\[
  a\le b \iff \exists c[a + c = b]
\]
This is a partial order with least element $0$
(see paragraph following \cite[5.18]{Tar49}).
Some examples of this partial order are as follows.
\begin{itemize}
  \item In $\N$ and $\R$,
    $\le$ coincides with the usual order.
  \item In the CA of measures on $X$,
    $\mu\le\nu$ iff $\mu(S)\le\nu(S)$ for every measurable $S\subset X$.
  \item In the CA induced by a $\sigma$-complete, $\sigma$-distributive lattice,
    $\le$ is the partial order induced by the lattice,
    i.e. $a\le b$ iff $a = a\wedge b$.
  \item For the class of cardinals,
    $\kappa\le\lambda$ iff there is an injection $\kappa\hra\lambda$,
    and the fact that this is a partial order is the Cantor-Schroeder-Bernstein theorem.
\end{itemize}
We say that $a\in A$ is the \textbf{meet} (resp. \textbf{join}) of a family $(a_i)_{i\in I}$,
denoted $\bigwedge a_i$ (resp. $\bigvee a_i$),
if it is the meet (resp. join) with respect to $\le$.
We write $a\perp b$ if $a\wedge b = 0$.

A homomorphism from a GCA $A$ to a GCA $B$ is a function $\phi:A\to B$ satisfying the following:
\begin{enumerate}
  \item If $a = b + c$, then $\phi(a) = \phi(b) + \phi(c)$.
  \item $\phi(0) = 0$.
  \item If $a = \sum_n a_n$, then $\phi(a) = \sum_n \phi(a_n)$.
\end{enumerate}

In this paper,
we will consider an \textbf{action} of a countable group $\Gamma$ on a GCA $A$,
i.e. a map $\Gamma\times A\to A$,
denoted $(\gamma,a)\mapsto \gamma a$,
satisfying the following:
\begin{enumerate}
  \item $\gamma(\delta a) = (\gamma\delta)a$.
  \item $1a = a$.
  \item For every $\gamma\in\Gamma$,
    the map $A\to A$ defined by $a\mapsto\gamma a$ is a homomorphism.
\end{enumerate}

A \textbf{$\Gamma$-GCA} is a GCA $A$ equipped with an action of a countable group $\Gamma$.
An element $a$ in $A$ is \textbf{$\Gamma$-invariant} if $\gamma a = a$ for every $\gamma\in\Gamma$.
We say that $a$ and $b$ in $A$ are \textbf{equidecomposable}
if there exist $(a_\gamma)_{\gamma\in\Gamma}$ in $A$ such that
$a = \sum_\gamma a_\gamma$ and $b = \sum_\gamma \gamma a_\gamma$.

The main theorem is as follows,
where a GCA $A$ is \textbf{cancellative} if for every $a,b\in A$,
if $a + b = a$, then $b = 0$.

\begin{thm}\label{ThmMain}
  Let $A$ be a cancellative $\Gamma$-GCA with binary meets,
  and let $\sim$ be an equivalence relation on $A$
  such that the following hold:
  \begin{enumerate}
    \item Equidecomposable elements are $\sim$-related.
    \item If $a\sim b$ and $a + c\sim b + d$,
      then $c\sim d$.
    \item If $a\sim b$ and $a\perp \gamma b$ for every $\gamma \in \Gamma$,
      then $a = 0$ (this implies $b = 0$, since $a\perp b$).
  \end{enumerate}
  Then $a\sim b$ iff $a$ and $b$ are equidecomposable.
\end{thm}

By a \textbf{finite measure} on a GCA $A$,
we mean a homomorphism from $A$ to $\R$.

\begin{cor}\label{CorMeas}
  Let $A$ be a $\Gamma$-GCA with countable joins and let $\mu$ and $\nu$ be finite measures on $A$.
  Then $\mu$ and $\nu$ agree on every $\Gamma$-invariant element of $A$ iff they are equidecomposable.
\end{cor}

We recover Thorisson's theorem by setting $A = \B(X)$ (under disjoint union).
\begin{cor}[Thorisson, {\cite[Theorem 1]{Tho96}}]
  Let $X$ be a standard Borel $\Gamma$-space and let $\mu$ and $\nu$ be finite Borel measures on $X$.
  Then $\mu$ and $\nu$ agree on every $\Gamma$-invariant subset of $X$ iff they are equidecomposable.
\end{cor}

We also obtain a criterion for equidecomposability of subsets of a probability space.
A \textbf{probability measure preserving (pmp)} $\Gamma$-action on a standard probabability space $(X,\mu)$
is an action of $\Gamma$ on $(X,\mu)$ by measure-preserving Borel automophisms.
\begin{cor}\label{CorSets}
  Let $(X,\mu)$ be a standard probability space with a pmp $\Gamma$-action and let $A,B\in\MALG(X,\mu)$.
  Then $A$ and $B$ agree on every $\Gamma$-invariant measure $\ll\mu$ iff they are equidecomposable.\footnotemark
\end{cor}
\footnotetext{
Ruiyuan (Ronnie) Chen has pointed out that this also follows from the Becker-Kechris comparability lemma,
see \cite[4.5.1]{BK96}}

This generalizes a well-known result (for instance, see \cite[7.10]{KM04})
which says that if $\mu$ is ergodic,
then $A$ and $B$ are equidecomposable iff $\mu(A) = \mu(B)$
(note that in this case,
$\mu$ is the only $\Gamma$-invariant measure $\ll\mu$).

\Cref{CorSets} will be obtained via a more general result about projections in von Neumann algebras;
see \Cref{ThmVna} below.

\begin{rmk}
  The original result in \cite{Tho96} is stated for actions of locally compact groups,
  but it is not clear how to formulate an analogous theorem in the setting of cardinal algebras.
\end{rmk}

\subsection*{Acknowledgments}
We would like to thank Alexander Kechris for introducing the author to Thorisson's theorem and suggesting the use of cardinal algebras.
We would also like to thank Ruiyuan (Ronnie) Chen for stating Thorisson's theorem in terms of equidecomposition,
as well as for making the author aware of the Becker-Kechris lemma.

\section{Preliminaries}
 
\begin{defn}
  Let $A$ be a GCA.
  An element $a\in A$ is \textbf{cancellative}
  (called \textbf{finite} in \cite{Tar49})
  if any of the following equivalent conditions hold
  (see \cite[4.19]{Tar49} for proof of equivalence):
  \begin{enumerate}
    \item If $a + b = a + c$,
      then $b = c$.
    \item If $a +  b = a$,
      then $b = 0$.
    \item If $a + b \le a + c$,
      then $b \le c$.
  \end{enumerate}
  A GCA $A$ is \textbf{cancellative} if every $a\in A$ is cancellative.
\end{defn}

\begin{defn}
  A \textbf{closure} of a GCA $A$ is a CA $\ol A$ containing $A$ such that the following hold:
  \begin{enumerate}
    \item If $a$ and $(a_n)_n$ are in $A$,
      then $a = \sum a_n$ in $A$ iff $a = \sum_n a_n$ in $\ol A$.
    \item $A$ generates $\ol A$,
      i.e. for every $b\in A$,
      there exist $(a_n)_n$ in $A$ such that $b = \sum a_n$.
  \end{enumerate}
\end{defn}

\begin{prop}[{\cite[7.8]{Tar49}}]
  Every GCA has a closure.
\end{prop}

\begin{eg}
  \leavevmode
  \begin{itemize}
    \item $\ol\N$ is the set of extended natural numbers $\{0,1,2,\ldots,\infty\}$.
    \item $\ol\R$ is the extended real line $[0,\infty]$.
  \end{itemize}
\end{eg}

The following is easy to verify.
\begin{prop}\label{Reflective}
  If $A$ is a GCA with closure $\ol A$ and $B$ is a CA,
  then every homomorphism $A\to B$ extends uniquely to a homomorphism $\ol A\to B$.
\end{prop}

\begin{rmk}
  This shows that the closure is left adjoint to the forgetful functor
  from the category of CAs to the category of GCAs,
  so in particular,
  the closure is unique up to isomorphism.
\end{rmk}

Let $\Hom(A,B)$ denote the set of all homomorphisms from $A$ to $B$.
\begin{thm}
  Let $A$ be a GCA.
  Then $\Hom(A,\R)$ is a cancellative GCA with binary meets
  (under pointwise addition).
\end{thm}

\begin{proof}
  By \cite[2.1]{Sho90},
  $\Hom(\ol A,\ol\R)$ is a CA with binary meets,
  so $\Hom(A,\ol \R)$ is also CA with binary meets,
  since it is isomorphic to $\Hom(\ol A,\ol\R)$ by \Cref{Reflective}.
  Thus since $\Hom(A,\R)$ is closed $\le$-downwards in $\Hom(A,\ol\R)$,
  it is a GCA by \cite[9.18(i)]{Tar49},
  and it has binary meets.
  The cancellativity of $\Hom(A,\R)$ follows immediately from cancellativity of $\R$.
\end{proof}

\section{The main theorem}

We turn to the proof of the main theorem.
\begin{proof}[Proof of \Cref{ThmMain}]
  Fix an enumeration $(\gamma_n)_n$ of $\Gamma$.
  Suppose $a\sim b$.
  We define sequences $(a_n)$ and $(b_n)$ recursively as follows.
  Let $a_0 = a$ and $b_0 = b$.
  For the inductive step,
  choose $a_{n+1}$ and $b_{n+1}$ such that
  \begin{align*}
    a_n & = a_{n+1} + a_n\wedge \gamma_n b_n \\
    b_n & = b_{n+1} + \gamma_n^{-1} a_n\wedge b_n
  \end{align*}
  By the remainder axiom,
  there are some $a_\infty$ and $b_\infty$ such that for any $n$,
  we have
  \begin{align*}
    a_n & = a_\infty + \sum_{i\ge n} a_i\wedge \gamma_i b_i \\
    b_n & = b_\infty + \sum_{i\ge n} \gamma_i^{-1} a_i\wedge b_i
  \end{align*}
  In particular,
  we have
  \begin{align*}
    a & = a_\infty + \sum_n a_n\wedge \gamma_n b_n \\
    b & = b_\infty + \sum_n \gamma_n^{-1} a_n\wedge b_n
  \end{align*}
  Thus to show that $a$ and $b$ are equidecomposable,
  it suffices to show that $a_\infty = b_\infty = 0$.

  Now $a_\infty\sim b_\infty$ by the first condition,
  since $a\sim b$ and $\sum_n a_n\wedge \gamma_n b_n\sim \sum_n \gamma_n^{-1} a_n\wedge b_n$
  (by equidecomposability).
  Now for any $n$,
  we have $b_n \ge b_\infty + \gamma_n^{-1} a_n\wedge b_n$,
  and thus
  \[
    \gamma_n b_n
    \ge \gamma_n b_\infty + a_n\wedge \gamma_n b_n
    \ge a_\infty\wedge \gamma_n b_\infty + a_n\wedge \gamma_n b_n
  \]
  We also have
  \[
    a_n
    \ge a_\infty + a_n\wedge \gamma_n b_n
    \ge a_\infty\wedge \gamma_n b_\infty + a_n\wedge \gamma_n b_n
  \]
  Thus
  \[
    a_n\wedge\gamma_n b_n
    \ge a_\infty\wedge \gamma_n b_\infty + a_n\wedge \gamma_n b_n
  \]
  Since $A$ is cancellative,
  we have $0\ge a_\infty\wedge\gamma_n b_\infty$,
  i.e. $a_\infty\wedge \gamma_n b_\infty = 0$.
  Thus $a_\infty\perp \gamma b_\infty$ for every $\gamma\in\Gamma$.
  Thus by our hypothesis,
  we have $a_\infty = 0$ and $b_\infty = 0$.
\end{proof}

\section{Applications}
We apply this theorem to prove \Cref{CorMeas}:

\begin{proof}
  Recall that $\Hom(A,\R)$ is a cancellative GCA with binary meets,
  and it has a $\Gamma$-action given by $(\gamma\mu)(a):=\mu(\gamma^{-1}a)$.
  Define the equivalence relation $\sim$ on $\Hom(A,\R)$
  by setting $\mu\sim\nu$ iff $\mu(a) = \nu(a)$ for every $\Gamma$-invariant $a\in A$.
  It suffices to check the conditions in \Cref{ThmMain}.
  Conditions 1 and 2 are clear.
  For condition 3,
  suppose that $\mu\sim\nu$ and $\mu\perp \gamma\nu$ for every $\gamma\in \Gamma$,
  and fix $a\in A$.
  We must show that $\mu(a) = 0$.
  By \cite[3.12]{Tar49},
  we have $\mu\perp\sum \gamma\nu$,
  and thus by \cite[1.14]{Sho90}
  (which is stated for CAs,
  but whose proof works without modification for GCAs),
  we can write $a = b + c$ with $\mu(b) = 0$ and $\qty(\sum \gamma\nu)(c) = 0$.
  Identifying $\nu$ with its extension $\ol A\to\ol\R$,
  we have $\nu(\sum\gamma c) = 0$.
  Thus $\nu(\bigvee \gamma c) = 0$,
  so since $\mu\sim\nu$,
  we have $\mu(\bigvee\gamma c) = 0$.
  Thus $\mu(c) = 0$,
  and thus $\mu(a) = \mu(b) + \mu(c) = 0$.
\end{proof}

Next we recall some notions from the theory of operator algebras.
A \textbf{von Neumann algebra} is a weakly closed $*$-subalgebra $M$ of $B(H)$ containing the identity.
An element $x\in M$ is \textbf{positive} if $x = yy^*$ for some $y\in M$,
and the set of positive elements is denoted $M_+$.
There is a partial order on $M$ defined by setting $x\le y$ iff $y-x$ is positive.
An element $p\in M$ is a \textbf{projection} if $p = p^* = p^2$,
and two projections $p$ and $q$ are \textbf{Murray-von Neumann equivalent},
written $p\sim_{\MvN} q$,
if there is some $u\in M$ such that $p = uu^*$ and $q = u^*u$.
Let $P(M)$ denote the set of projections in $M$.
Then $P(M)/{\sim}_{\MvN}$ is a lattice.
A projection $p$ is \textbf{finite} if for any projection $p'$,
if $p\sim_{\MvN} p'\le p$,
then $p = p'$.
A von Neumann algebra $M$ is \textbf{finite} if $1_M$ is a finite projection.
A \textbf{trace} on $M$ is a map $\tau:M_+\to\ol\R$ such that $\tau(mm^*) = \tau(m^*m)$,
and a trace is \textbf{finite} if its image is contained in $\R$.
A trace is \textbf{faithful} if $\tau(m) = 0$ implies $m = 0$,
and a trace is \textbf{normal} if it is weakly continuous.

If $M$ is a von Neumann algebra,
then $P(M)/{\sim}_{\MvN}$ is a GCA under join of orthogonal projections \cite{Fil65},
and if $M$ is finite,
then this GCA is cancellative.
A \textbf{$\Gamma$-action} on a von Neumann algebra $M$
is an action of $\Gamma$ on $M$ by weakly continuous $(+,0,\cdot,1,*)$-homomorphisms.
Every von Neumann algebra $M$ with a $\Gamma$-action gives rise to a $\Gamma$-GCA,
and a trace $\tau$ on $M$ is said to be \textbf{$\Gamma$-invariant}
if $\tau(\gamma m) = \tau(m)$ for every $m\in M$ and $\gamma\in\Gamma$.

\begin{thm}\label{ThmVna}
  Let $M$ be a finite von Neumann algebra with a $\Gamma$-action
  which admits a faithful normal finite $\Gamma$-invariant trace,
  and let $[p],[q]\in P(M)/{\sim}_{\MvN}$.
  Then $[p]$ and $[q]$ agree on every finite $\Gamma$-invariant trace on $M$
  iff they are equidecomposable.
\end{thm}

\begin{proof}
  Let $A = P(M)/{\sim}_\MvN$,
  which is a cancellative $\Gamma$-GCA with binary meets.
  Now define the equivalence relation on $A$ by setting
  $[p]\sim[q]$ if $[p]$ and $[q]$ agree on every $\Gamma$-invariant trace on $M$.
  It suffices to check the conditions in \Cref{ThmMain}.
  Conditions 1 and 2 are clear.
  For condition 3,
  suppose that $[p]\sim[q]$ and $[p]\perp\gamma[q]$ for every $\gamma\in\Gamma$,
  and fix a faithful normal finite $\Gamma$-invariant trace $\tau$ on $M$.
  Then setting $\ol p = \bigvee\gamma p$,
  the map $m\mapsto\tau(\ol p m\ol p)$ is a finite $\Gamma$-invariant trace on $M$.
  Since $\tau(\ol pq\ol p) = 0$ and $p\sim q$,
  we have $\tau(p) = \tau(\ol pp\ol p) = 0$.
  Thus $p = 0$.
\end{proof}

\Cref{CorSets} follows by applying this to $L^\infty(X,\mu)$.
\begin{proof}[Proof of \Cref{CorSets}]
  Let $M = L^\infty(X,\mu)$.
  This is a finite $\Gamma$-von Neumann algebra and
  $\mu$ induces a faithful normal finite $\Gamma$-invariant trace on $M$.
  Now $P(M)/{\sim}_{\MvN}$ is isomorphic to $\MALG(X,\mu)$ as a lattice (with $\Gamma$-action),
  so they give rise to isomorphic $\Gamma$-GCAs,
  and thus we are done by \Cref{ThmVna}.
\end{proof}

\bibliography{equidecomp_cardalg}
\bibliographystyle{abstract}
\end{document}